\documentclass[11pt]{article}   
\usepackage[colorlinks,citecolor=black]{hyperref}
\hypersetup{linkcolor=black}
\usepackage{graphicx}
\usepackage{enumerate}
\usepackage{amssymb, amsmath, amsthm}
\usepackage[numbers]{natbib}
\usepackage{graphicx}
\usepackage{subfig}
\usepackage{epstopdf}

\newtheorem{proposition}{Proposition}[section] 

\newtheorem{thm}[proposition]{Theorem}

\newtheorem{corollary}[proposition]{Corollary}
\newtheorem{remark}[proposition]{Remark}

\newtheorem{example}[proposition]{Example}

\usepackage[T1]{fontenc}
\usepackage{lmodern}

\newcommand{\R}{\ensuremath{{\mathbb R}}}

\newcommand{\e}{\ensuremath{{\rm e}}}
\newcommand{\ch}{\ensuremath{{\rm ch}}}
\newcommand{\sh}{\ensuremath{{\rm sh}}}

\newcommand{\E}{\ensuremath{{\mathbb E}}}
\newcommand{\PP}{\ensuremath{{\mathbb P}}}

\begin{document}
\title{Differentiablity of excessive functions of one-dimensional
  diffusions and the principle of smooth fit }
\author{Paavo Salminen\;\thanks{\AA bo Akademi, Department of Natural Sciences,
    Mathematics and Statistics, \small FIN-20500 \AA bo, Finland, e-mail: 
phsalmin@abo.fi, tbao@abo.fi}
  \thanks{Research supported in part by a grant
      from Svenska kulturfonden via Stiftelsernas professorspool,
      Finland} \quad and\quad Bao Quoc Ta\footnotemark[1]\;\thanks{Research supported by the Finnish Doctoral Programme on Stochastics and Statistics}}
  \date{}
  \maketitle
\begin{center}
\it Dedicated to the memory of a dear friend and colleague\\ Professor Esko Valkeila
(1951-2012)
\end{center}

\vspace{0.5cm}

\begin{abstract}
The principle of smooth fit is probably the most used tool to find solutions to optimal stopping problems of 
one-dimensional diffusions. It is important, e.g., in financial mathematical applications to understand in which kind of models 
and problems smooth fit can fail. In this paper we connect - in case of one-dimensional 
diffusions - the validity of smooth fit and the differentiability of excessive functions. The basic
tool to derive the results is the representation theory of excessive
functions; in particular, the Riesz and Martin representations. It is seen that the
differentiability may not hold  in case the speed measure of the diffusion or the
representing measure of the excessive function has atoms. 

As an example, we study  optimal stopping of sticky
Brownian motion. It is known that the validity of the smooth fit in this case  depends on the value of the discounting parameter (when the other parameters are fixed). We decompose the size of the jump in the derivative of the value function into two factors. The first one is due to the atom of the representing measure and the second one due to the atom of the speed measure.
 \end{abstract}
  
  \section{Introduction}
Let $X=(X_t)_{t\geq 0}$ be a one-dimensional 
diffusion process  in the sense of It\^o and McKean
\cite{itomckean74} 
living on an interval $I\subseteqq \R$ , i.e., $X$ is a time-homogeneous strong Markov
process with continuous sample paths.
As usual, the notations   $\PP_x$
and $\E_x$ are used  for the
probability measure and the expectation operator, respectively, associated with $X$ when initiated
from $x\in I$. The life time of $X$ is defined as $\zeta:=\inf\{t:
X_t\not\in I\}$ and we set $X_t=\Delta$ for $t\geq \zeta$, where
$\Delta$ is a fictitious state -- the so called cemetery state. 
Recall that a measurable function $f: I\cup{\{\Delta\}} \mapsto \R_+\cup\{\infty\}$ is
called $\alpha$-excessive, $\alpha\geq 0,$ if
for all $ x\in I$ the following two conditions hold:
\begin{align}
&
\label{excessive1}
\E_x (e^{-\alpha t}f(X_t))\leq f(x),\quad \forall t>0,\\
&
\label{excessive2}
\lim_{t\downarrow 0}\E_x (e^{-\alpha t}f(X_t))= f(x),
\end{align}
where, by convention, $f(\Delta)=0.$ An alternative and equivalent
definition is obtained by replacing (\ref{excessive1}) and
(\ref{excessive2}) by 
\begin{align}
&
\label{dynkin1}
\beta\,  \E_x \left(\int_0^\zeta {\rm e}^{-(\alpha+\beta)
  t}f(X_t)\,dt\right)\leq f(x),\quad \forall \beta >0,\\
&
\label{dynkin2}
\lim_{\beta\uparrow+\infty}\E_x \left(\int_0^\zeta {\rm e}^{-(\alpha+\beta) t}f(X_t)\,dt\right)=f(x).
\end{align} 
We refer to  
Dynkin \cite{dynkin65} Vol. II for results on excessive functions in
general and in particular for one-dimensional diffusions. For excessive
functions in the framework of the potential theory of Markov processes,
see Blumenthal and Getoor \cite{blumenthalgetoor68} and Chung and Walsh
\cite{chungwalsh05}. 
Excessive functions being descendants of superharmonic
functions have, hence, deep roots in the classical 
potential theory and constitute also fundamental concept in the theory of
Markov processes.  

Our main motivation for the present study comes, however, 
from the theory of optimal stopping where excessive functions play a
crucial role. Indeed, given a continuous non-negative (reward) function $g$
the
optimal stopping problem with the underlying process $X$ is to find a
(value) function $V$ and an (optimal) stopping time $\tau^*$ such that 
\begin{equation}\label{osp}
V(x):=\sup_{\tau\in {\cal
    M}}\E_x(e^{-\alpha\tau}g(X_\tau))=\E_x(e^{-\alpha\tau^*}g(X_{\tau^*})),
\end{equation}
where $\cal M$ denotes the set of all stopping times with respect to the
filtration $({\cal F}_t)_{t\geq 0}$ generated by $X.$ The fundamental
result due to Snell and Dynkin (see Shiryayev \cite{shiryayev78} 
and Peskir and Shiryayev \cite{peskirshiryaev06} for details and references),
is that
$V$ is the smallest $\alpha$-excessive function dominating $g$ and an
optimal stopping time is given by
$$
\tau^*=\inf\{t\geq 0: X_t\in\Gamma\},
$$    
where  $\Gamma:=\{x: V(x)=g(x)\}$ is the so called stopping region. Therefore, 
a good knowledge of excessive functions is a key
to a deeper understanding of optimal stopping. 

Although our focus is on applications in optimal stopping we wish to
point out that excessive functions can also be used, e.g.,  to condition and/or to kill a process
in some particular desirable way. Such conditionings of the underlying
process are called
excessive transforms or Doob's $h$-transforms due to Doob's pioneering
work  \cite{doob57}. We refer also to McKean  \cite{mckean63},
Dynkin \cite{dynkin69_engl}, and  Meyer et
al. \cite{meyersmythewalsh71} for early seminal papers. The
theory of  $h$-transforms in a general setting is discussed in Chung and Walsh \cite{chungwalsh05} Chapter 11.  
Moreover, a fairly recent problem arising from financial mathematics is to construct
for a given process $X$   a
martingale having the same distribution as $X$ at a fixed time 
or at a random time, see, e.g.,  Cox et al. \cite{coxetal11},  Hirsch et al \cite{hirschetal11}, Ekstr\"om
et al. \cite{ekstrometal13} and Noble \cite{noble13} and references therein. In particular,
Klimmek \cite{klimmek12} exploits explicitly $h$-transforms to find
the solution of the problem for a random (exponential) time.

These and other applications in mind -- and also per se --  we offer in
this paper firstly a discussion on continuity and differentiablity properties of
excessive functions of one-dimensional diffusions and secondly
applications to optimal stopping with an example. Our approach
utilizes the Riesz and Martin representations  which are valid  in their strongest and most explicit
forms for
one-dimensional regular diffusions.

In the next section we give the Riesz and Martin representations
with needed prerequisities and also present some examples. 
An
immediate implication of the Riesz representation is then  the continuity of
excessive functions, see Proposition \ref{cont}. 
The continuity is implicitly stated already in Salminen \cite{salminen85}. In  Dayanik and Karatzas
\citep{D-K} and in Peskir and Shiryayev  \cite{peskirshiryaev06} the 
continuity is proved for a special class of
excessive functions, that is, for value functions in optimal
stopping problems. Their proofs utilize the properties of the value
functions and the concavity, as in Dynkin and Yuschkevitch
\cite{dynkinyushkevich69}. The advantage of the present approach is
that it yields the result with full generality. We also shortly
list - for general interest - some other basic potential theoretical and related results for one-dimensional
regular diffusions. In particular, it is seen that all additive
functionals are continuous. This is also pointed out in \cite{borodinsalminen02}
p. 28 but with a slightly different explanation.

In the third section the differentiability properties of excessive
functions are investigated. The Riesz representation allows us to derive
conditions for differentiablity with respect to any increasing
continuous function $F,$ see Theorem \ref{exc-thm}.  This extends the result in  \cite{salminen85}
where differentiability with respect to the scale function is studied.
We also  represent the jump of the derivative of an excessive function
as the sum of two terms: the first one is induced by the
representing measure and the second one by the speed measure.

These results are then used in the fourth
section to
study the principle of smooth fit in optimal stopping of 
one-dimensional diffusions. Our contribution hereby is to demonstrate
-- using the results in Section 3  -- that
the proof of the condition for the smooth fit with respect to the
scale as presented in  \cite{salminen85} can be rewritten -- 
changing mainly only the notation -- to a proof of the condition for the smooth fit
in the ordinary sense as given in  Peskir
\cite{Pes} and  Samee \cite{Sam}, see also  \cite[p. 160]{peskirshiryaev06} (e.g. when studying the case where the scale function is not
differentiable at the stopping point).  We conclude by analyzing the smooth fit
property in an optimal stopping
problem where the underlying process is a sticky
Brownian motion. It is known, see Crocce and Mordecki \cite{croccemordecki12}, that if the optimal stopping
point is the sticky point then the smooth fit typically fails. Our
results enhance the understanding of this phenomenon by giving an
explicit form for the jump of the derivative of the value function in
this case.

\section{Riesz and Martin representations}

We start with by introducing more notation and recalling some basic
facts. Let $l\geq -\infty$ and $r\leq +\infty$ denote the left and the right, respectively,
end point of $I$ which is an interval of any kind. Recall that  $I$ is the state space of $X$.  The notations $m$ and $S$
are used for the speed measure and the scale function, respectively. Moreover, let $\cal G$ denote the generalized differential operator
associated with $X$ and $\tau_y$ the
first hitting time of $y\in I$, that is, 
$$
\tau_y:=\inf\{t: X_t=y\}.
$$
We assume that $X$ is regular (cf. Dynkin \cite{dynkin65} Vol. II p.121), that is, 
\begin{equation}
\label{regu}
 \PP_x(\tau_y<\infty)>0,\qquad  \forall x,y\in I
\end{equation}
in other words, no matter where $X$ starts there is a positive probability to
hit any point in $I.$ This means that a regular diffusions do not have
absorbing points and
exit points and entrance points are not included in $I.$ Moreover, a
consequence of the regularity is that there does not exist non-empty
polar sets (for this notion, see \cite[p. 79]{blumenthalgetoor68} ). 
\begin{remark}
\label{other_regu}
The above definition of regularity differs  from another often used definition in which (\ref{regu}) is
assumed to hold  for all $x\in (l,r)$ and $y\in I$ (see,  e.g., Revuz
and Yor \citep{R-Y} p. 300). According to this latter definition we
could have a regular diffusion with $I=[l,r]$ and  $l$ and $r$ absorbing. As demonstrated below
in Example \ref{absorbing} such a diffusion has discontinuous excessive
functions  -- the case we want to exclude.
\end{remark}

As showed in It\^o and McKean \cite[p. 124]{itomckean74}  the Laplace
transform of $\tau_y$ can be expressed for  $\alpha>0$ as
\begin{equation*}
\E_x(e^{-\alpha\tau_y})=
\left\{
\begin{array}{rl}
\displaystyle{\frac{\psi_\alpha(x)}{\psi_\alpha(y)}},\quad x\leq y,\\
\displaystyle{\frac{\varphi_\alpha(x)}{\varphi_\alpha(y)}},\quad x\geq y,
\end{array}\right. 
\end{equation*}
where $\psi_\alpha$ and $\varphi_\alpha$ are continuous, positive,  increasing and
decreasing, respectively, solutions of the generalized differential
equation  
\begin{equation}
\label{GEE}
{\cal G}u=\alpha u.
\end{equation}
Imposing appropriate boundary conditions 
determine $\psi_\alpha$ and $\varphi_\alpha$ uniquely up to a multiplicative constant. The Wronskian $\omega_{\alpha}$ - a constant - is defined as  
\begin{align*}
\omega_{\alpha}&:=\psi^{+}_{\alpha}(x)\varphi_\alpha(x)-\psi_\alpha(x)\varphi^{+}_{\alpha}(x)\\
&=\psi^{-}_{\alpha}(x)\varphi_\alpha(x)-\psi_\alpha(x)\varphi^{-}_{\alpha}(x),
\end{align*} 
where the superscripts $^+$ and $^-$ denote the right and left
derivatives with respect to the scale
function, i.e., for $u=\psi_\alpha$ or $\varphi_\alpha$ 
\begin{align*}
u^+(x):=\frac{d^+u}{dS}(x):=&\lim_{\delta\to 0+} \frac{u(x+\delta)-u(x)}{S(x+\delta)-S(x)},
\\
u^-(x):=\frac{d^-u}{dS}(x):=&\lim_{\delta\to 0+} \frac{u(x-\delta)-u(x)}{S(x-\delta)-S(x)},
\end{align*}
cf. (\ref{derivatives1}) and (\ref{derivatives2}) below and recall that the scale function of a diffusion is continuous.    
It is well-known (see \cite[p. 150]{itomckean74} ) that 
\begin{equation}\label{greenkernel}
G_\alpha(x,y):=
\left\{
\begin{array}{rl}
w^{-1}_\alpha\psi_\alpha(x)\varphi_\alpha(y),\quad x\leq y,\\
w^{-1}_\alpha\psi_\alpha(y)\varphi_\alpha(x),\quad x\geq y,
\end{array}\right.
\end{equation}
serves as a resolvent kernel (also called the Green function) of $X,$ i.e., for any Borel subset $A$ of $I$ 
$$
\E_x\left(\int_0^\zeta {\rm e}^{-\alpha t}{\bf 1}_A(X_t)\, dt\right) = \int_A G_\alpha(x,y)\,m(dy),
$$ 

Using Theorem 12.4  in Dynkin
\cite{dynkin65} it is fairly straightforward to check
that for every fixed $y$ the function $x\mapsto G_\alpha(x,y)$ is
$\alpha$-excessive (see \cite[p. 89]{salminen85} ). 
Since $(x,y)\mapsto G_\alpha(x,y)$ is symmetric it follows that
$X$ is self-dual with respect to the speed measure, that is
\begin{equation}\label{dual}
<f, G_\alpha g>_m\,=\,< G_\alpha f,g>_m,
\end{equation}
where  
$$
< f,g>_m:= \int_{I}  f(x) g(x)m(dx),\qquad 
 G_\alpha f(x):=\int_{I}  G_\alpha(x,y)f(y)m(dy),
$$
with $f$ and $g$ bounded Borel measurable functions satisfying appropriate
integrability condtions. For the concept of duality and related
topics,  see Kunita and Watanabe \cite{KW}, Blumenthal and Getoor
\cite{blumenthalgetoor68} and Chung and Walsh
\cite{chungwalsh05}. We wish to apply  the 
Riesz representation theorem, see \cite[p. 272]{blumenthalgetoor68}, 
and remark that the assumptions for its validity as presented in 
 \cite{blumenthalgetoor68}
Chapter VI (see also \cite[Theorem 2 p. 505]{KW}) are satisfied. An
important assumption is that $X$ has a dual process which is standard  in the sense
of the definition in  ibid. p. 45. Clearly, $X$ is standard and since
$X$ is self dual the needed assumption is fulfilled. Notice also that  (2.1) and (2.2) in  \cite[p. 265-266]{blumenthalgetoor68}
hold.

\begin{thm}
\label{RDEC}
(The Riesz representation) Let $\alpha>0$ and $u$ an 
  $\alpha$-excessive function of the regular one-dimensional diffusion $X$. It is assumed that $u$ is locally integrable with respect
  to $m.$ Then there exist
an  $\alpha$-harmonic function $h_\alpha$ and a Radon measure $\sigma_u$
on $I$ such that $u$ can be represented uniquely as 
\begin{equation}
\label{riesz} 
u(x)=\int_{I} G_\alpha(x,y)\, \sigma_u(dy) + h_\alpha(x).
\end{equation}
\end{thm}
\noindent
\begin{remark} {\bf (i)}  The Riesz representation holds also for
  $\alpha=0$ when $X$ is transient. The Green function when
   $\alpha=0$ has a
  similar structure as in case $\alpha>0$ but now the corresponding 
functions  $\varphi_0$ and $\psi_0$ express the hitting probabilities
instead of the Laplace transforms of the hitting distributions. In
Section \ref{DIFF} we discuss shortly the special case in which the
diffusion is not killed inside the state space $I.$ \hfill\break\hfill
 {\bf (ii)} The assumption on the local integrability is
  superfluous in case of one-dimensional diffusions. Indeed, assuming that the $\alpha$-excessive function $u\not\equiv +\infty$ choose a point $x$ such
  that $u(x)<+\infty.$ From (\ref{dynkin1}) we have 
\begin{equation}
\beta \int_I G_{\alpha+\beta}(x,y)u(y)m(dy)=
\beta\,  \E_x \left(\int_0^\zeta {\rm e}^{-(\alpha+\beta)
  t}u(X_t)\,dt\right)\leq u(x).
\end{equation}
The local integrability of $u$ follows now easily from the explicit
form of  $G_{\alpha+\beta}$  and the continuity of $\psi_{\alpha+\beta}$
and  $\varphi_{\alpha+\beta}$.  \hfill\break\hfill
 {\bf (iii)} The $\alpha$-harmonicity of $h_\alpha$ means that 
for all compact subset $A$ of $I$ it holds  
\begin{equation}
\label{harm} 
h_\alpha(x)=\E_x\left( {\rm e}^{-\alpha\tau_A}\,h_\alpha(X_{\tau_A})\right),
\end{equation}
 where 
$$ 
\tau_A:=\inf\{t\,:\, X_t\not\in A\}.
$$
\end{remark}

The Riesz representation does not give much information on the harmonic function associated with a given excessive function. 
However, in the Martin boundary theory an integral representation is derived also for the harmonic functions. We refer to
\cite{KW} and \cite{chungwalsh05} for the general theory of Martin boundaries for Markov processes.
In the next theorem we state the Martin representation for
one-dimensional regular diffusions and, moreover, present the explicit form 
of the representing measure extracted from \cite{salminen85}. We refer also to  \cite{alvarezsalminen} and \cite{christensenirle11} for applications of the Martin boundary 
theory in optimal stopping.  

\begin{thm}
\label{MDEC}
(The Martin representation) Let $u$ be an 
  $\alpha$-excessive function of the one-dimensional diffusion $X$ and $x_o\in I$ a point such that
$u(x_o)=1.$ Then $u$ can be represented uniquely as 
\begin{equation}\label{martin-repre}
u(x)=\int_{(l,r)}\frac{G_{\alpha}(x,y)}{G_{\alpha}(x_0,y)}\nu^o_{u}(dy)
+\frac{\varphi_\alpha(x)}{\varphi_\alpha(x_o)}\,\nu^o_{u}(\{l\})+\frac{\psi_\alpha(x)}{\psi_\alpha(x_o)}\,\nu^o_{u}(\{r\}),
\end{equation}
where  $\nu^o_{u}$ is a probability measure on $[l,r]$ characterized via 
\begin{align}
\label{10}
\nu^o_{u}((x,r])&=\frac{\psi_\alpha(x_o)}{\omega_{\alpha}}
\left(\varphi_\alpha(x)\,u^+(x)-u(x)\,\varphi^+_\alpha(x)\right),\qquad x\geq x_o,\\
\label{11}
\nu^o_{u}([l,x))&=\frac{\varphi_\alpha(x_o)}{\omega_{\alpha}}
\left(u(x)\,\psi^-_\alpha(x)-\psi_\alpha(x)\,u^-(x)\right),\qquad x\leq x_o.
\end{align}
Conversely, given a probability measure $\mu$ on $[l,r]$ and  $x_o\in
I$ then the right hand side of (\ref{martin-repre}) when putting
$\nu^o_{u}=\mu$ defines an $\alpha$-excessive function.
\end{thm}
\begin{remark} The expression on the right hand side of
  (\ref{martin-repre}) is well defined since $\varphi_\alpha$ and
  $\psi_\alpha$ are positive on $I.$ Notice that the probability
  measure $\nu^o_{u}$ is defined on the closure of $I;$ also in case
  $l=-\infty$ and/or $r=+\infty.$ In fact, $[l,r]$ is the
  so called Martin
  compactification of $I.$

\end{remark}

Combining (\ref{martin-repre}) with (\ref{riesz}) yields a
characterization of the $\alpha$-excessive functions in the Riesz
representation. This together with other relationships between the two
representations are discussed in the next  
\begin{corollary}
Let $u$ and $h_\alpha$ be as in Theorem \ref{RDEC}. Then there exist
$c_1\geq 0$ and $c_2\geq 0$ such that 
$h_\alpha =c_1\, \varphi_\alpha+ c_2\,\psi_\alpha.$ The Riesz and the
Martin representing measures of $u$ are connected via the identity
\begin{equation}\label{martin-riesz}
\sigma_{u}(A)=\int_A\frac{1}{G_{\alpha}(x_0, y)}\,\nu^{o}_{u}(dy),
\end{equation}
where $A$ is a Borel subset of $I.$ Moreover, 
$u$ has the unique representation 
\begin{equation}
\label{riesz2} 
u(x)=\int_{(l,r)} G_\alpha(x,y)\, \sigma_u(dy) + \hat h_\alpha(x),
\end{equation}
where 
$$
\hat h_\alpha(x):=c'_1\, \varphi_\alpha(x)+ c'_2\,\psi_\alpha(x)
$$
with $c_1'\geq 0$ and $c'_2 \geq 0.$
\end{corollary}

Next example highlights the difference of the Riesz and Martin
representations via the fact that  $\psi_\alpha$ and/or
$\varphi_\alpha$ could be potentials, that is, not $\alpha$-harmonic.

\begin{example} Let $X$ be a Brownian motion reflected at 0 and killed
  at 1. Hence, $I=[0,1)$ and it is readily checked that we may take  
$$
\varphi_\alpha(x)={\rm sh}((1-x)\sqrt{2\alpha})
\quad {\rm and}\quad
\psi_\alpha(x)={\rm ch}(x\sqrt{2\alpha}).
$$
Notice that $\varphi_\alpha(1)=0$ and $\psi'_\alpha(0+)=0$ which, in fact, are
the appropriate boundary conditions to characterize $\varphi_\alpha$
and $\psi_\alpha,$ respectively. For this process, $\psi_\alpha$ is
$\alpha$-harmonic but $\varphi_\alpha$ is not. Standard computations
show that both  $\psi_\alpha$ and $\varphi_\alpha$ satisfy the
$\alpha$-harmonicity condition  (\ref{harm}) for intervals of the form 
$[a,b],\ 0<a<b<1.$ However, when $a=0$ the condition fails for
$\varphi_\alpha.$ Indeed, putting $A=[0,b],\ 0<b<1$ we have  
\begin{align*}
\E_x(e^{-\alpha\tau_{A}}\varphi_{\alpha}(X_{\tau_A}))&=\E_x(e^{-\alpha\tau_{b}}\varphi_{\alpha}(X_{\tau_b}))\\
&=\varphi_{\alpha}(a)\E_x(e^{-\alpha\tau_{a}})\\
&=\varphi_{\alpha}(b)\frac{\psi_{\alpha}(x)}{\psi_{\alpha}(b)}\neq \varphi_{\alpha}(x).
\end{align*}
Consequently, the Riesz representation of $\varphi_{\alpha}$ does not
have the $\alpha$- harmonic part and, hence, 
\begin{align*}
\varphi_{\alpha}(x)&=\int_{[0,1)}G_{\alpha}(x,y)\sigma_\varphi(dy).
\end{align*}
It follows by the uniqueness of the representing measure that $\sigma_\varphi$ is a multiple of the Dirac measure at 0.
\end{example}

Next we prove the continuity of excessive functions which is an
important stepping stone to the differentiability studied in Section
3.    

\begin{proposition}
\label{cont}
For a one-dimensional regular diffusion all $\alpha$-excessive
functions are continuous.
\end{proposition}
\begin{proof}
Let $u$ be an  $\alpha$-excessive function. Substituting the explicit
form of the Green kernel in the 
representation (\ref{riesz2}) yields 
\begin{align*}
u(x)=& w_\alpha^{-1}\varphi_{\alpha}(x)\int_{(l,
  x]}\psi_{\alpha}d\sigma_u+w_\alpha^{-1}\psi_{\alpha}(x)\int_{(x,
    r)}\varphi_{\alpha}d\sigma_u + \hat h_\alpha(x)
\\
=&
 w_\alpha^{-1}\varphi_{\alpha}(x)\int_{(l,
  x)}\psi_{\alpha}d\sigma_u+w_\alpha^{-1}\psi_{\alpha}(x)\int_{[x,
    r)}\varphi_{\alpha}d\sigma_u + \hat h_\alpha(x),
\end{align*}
from which evoking the continuity of $\varphi_{\alpha}$ and
$\psi_\alpha$ it easily follows that
$$
\lim_{\varepsilon\to 0+}u(x+\varepsilon)=\lim_{\varepsilon\to
  0+}u(x-\varepsilon)=u(x).
$$
\end{proof}
\begin{example}
\label{absorbing}
To stress the importance of the regularity as defined in (\ref{regu})
we give an example showing that if an end point of I is absorbing then
there exist discontinuous excessive functions. Let $X$ denote a
Brownian motion on $\R^+$ absorbed at 0 and consider the function
\begin{equation*}
f(x)=\begin{cases}
1, \quad\text{if}\quad x>0,\\
\frac{1}{2},\quad\text{if}\quad x=0.
\end{cases}
\end{equation*}
Since $\PP_0(X_t=0)=1$ for all $t\geq 0$ we have for 
$x>0$
\begin{align*}
\E_x(f(X_t))
&=\PP_x(t<\tau_0)+\frac{1}{2}\,\PP_x(t\geq\tau_0)\\
&=1-\frac{1}{2}\,\PP_x(t\geq\tau_0)\\
&\leq 1=f(x),
\end{align*} 
and, for $x=0$,
$$
\E_0(f(X_t))=\frac{1}{2}= f(0).
$$
Clearly, also (\ref{excessive2}) holds. Consequently, $f$ is a
discontinuous excessive function.
\end{example}



We conclude this subsection by pointing out some important properties of
one-dimensional diffusions which can be deduced from the potential
theoretical generalities using the explicit form of the Green kernel 
 and the continuity of
the excessive functions: 
\hfill\break\hfill
($\bullet$) Firstly, since $X$ is self dual and for all
$y\in I$ the function $x\mapsto G_\alpha(x,y)$ is bounded and
continuous it follows from (4.11) p. 290 in \cite{blumenthalgetoor68} that every point
in $I$ is regular, i.e.,  
\begin{equation}
\label{regu1}
\PP_x(\tau^+_x=0)=1\quad \forall x\in I,
\end{equation}
where 
$$
\tau^+_x:=\inf\{t>0: X_t=x\}.
$$
Consequently, by ibid. (3.13) p. 216, $X$ posseses  at every
point $x\in I$ a local time
$(L^{(x)}_t)_{t\geq 0}.$ 
\hfill\break\hfill
($\bullet$) Secondly, again due to the self duality together with the continuity
of the sample paths, it holds  that all additive functionals of
$X$ are continuous, see ibid. p. 289. In particular, $t\mapsto L^{(x)}_t$ is a.s. continuous.
\hfill\break\hfill
($\bullet$) Thirdly, recall that an excessive function $f$  is called regular for
  $X$ if $t\mapsto f(X_t)$ is continuous on $[0,\zeta)$ (see
    ibid.  p. 287-288). Hence, from Proposition \ref{cont}
 it follows via the continuity of the sample paths that
    all (finite) excessive functions for $X$ are regular.

\section{Differentiability}
\label{DIFF}
For an increasing continuous function $F:I\mapsto \R$ and an arbitrary
$\alpha$-excessive function $u$ we introduce the one sided derivatives
of $u$ with respect to $F$:
\begin{align}
\label{derivatives1}
\frac{d^+u}{dF}(x):=&\lim_{\delta\to 0+} \frac{u(x+\delta)-u(x)}{F(x+\delta)-F(x)},
\\
\label{derivatives2}
\frac{d^-u}{dF}(x):=&\lim_{\delta\to 0+} \frac{u(x-\delta)-u(x)}{F(x-\delta)-F(x)}
\end{align}
for every $x\in I$ for which the limits on the right hand sides exist and are finite. We say that
$u$ is $F$-differentiable at $x\in I$ if
\begin{equation*}
\frac{d^+u}{dF}(x)= \frac{d^-u}{dF}(x).
\end{equation*}
Our basic result gives conditions for the $F$-differentiability of an
arbitrary $\alpha$-excessive function $u$. Recall from the Riesz
representation that there exists a Radon measure $\sigma_u$ such that 
(\ref{riesz}) holds.
\begin{thm}
\label{exc-thm}
Let $u$ and $F$ be as above and assume that functions $\psi_\alpha$ 
and $\varphi_\alpha$
 are for $\alpha>0$ $F$-differentiable
 at a point $z\in
I$. Then the left and the right $F$-derivative of $u$ exist at $z$ and
satisfy
\begin{equation}\label{diff-exc}
\frac{d^- u}{dF}(z)-\frac{d^+u}{dF}(z)\geq 0.
\end{equation}
Moreover, $u$ is $F$-differentiable at $z$  if and only if $\sigma_u(\{z\})=0$.
\end{thm}
\begin{proof}
Since functions  $\psi_\alpha$ and $\varphi_\alpha$ are assumed to be
$F$-differentiable at $z$ we may, without loss of generality, take
$c_1'=c_2'=0$ in (\ref{riesz2}), and, hence, $u$ has the representation

\begin{equation}{\label{riesz3}}
u(z)=w_\alpha^{-1}\varphi_\alpha(z)\int_{(l,z]}\psi_\alpha d\sigma_u+w_\alpha^{-1}
\psi_\alpha(z)\int_{(z,r)}\varphi_\alpha d\sigma_u.
\end{equation}
Using (\ref{riesz3}) it is seen after som simple manipulations that for $\delta>0$
\begin{align*}
\hskip-1cm\frac{u(z+\delta)-u(z)}{F(z+\delta)-F(z)}&=w_\alpha^{-1}\,\frac{\varphi_\alpha(z+\delta)
-\varphi_\alpha(z)}{F(z+\delta)-F(z)}\,\int_{(l,z]}\psi_\alpha d\sigma_u\\
&\quad\quad +w_\alpha^{-1}\,\frac{\psi_\alpha(z+\delta)-\psi_\alpha(z)}{F(z+\delta)-F(z)}\,\int_{(z,r)}\varphi_\alpha d\sigma_u\\
&\quad\quad\quad + w_\alpha^{-1}J(z,\delta),
\end{align*}
where 

\[J(z,\delta):=\frac{\varphi_\alpha(z+\delta)\int_{(z, z+\delta]}\psi_\alpha d\sigma_u
 -\psi_\alpha(z+\delta)\int_{(z, z+\delta]}\varphi_\alpha d\sigma_u}{F(z+\delta)-F(z)}.\]
Since $\psi_\alpha$ and $\varphi_\alpha$ are increasing and
decreasing, respectively, it holds
\[R(z,\delta)\leq J(z,\delta)\leq 0\]

\begin{eqnarray*}
&&\hskip-.5cm
R(z,\delta):=
\frac{\varphi_\alpha(z+\delta)\psi_\alpha(z)-\psi_\alpha(z+\delta)
\varphi_\alpha(z)}{F(z+\delta)-F(z)}\ 
\sigma_u((z,z+\delta]).
\end{eqnarray*}
Evoking that  $\sigma_u$ is a measure and  $\varphi_\alpha$ and $\psi_\alpha$
are assumed to be $F$-differentiable at $z$ we obtain 
\begin{align*}
\lim_{\delta\downarrow
  0}R(z,\delta)&=\big(\frac{d\varphi_\alpha}{dF}(z)\psi_\alpha(z)-
\varphi_\alpha(z)\frac{d\psi_\alpha}{dF}(z)\big)\lim_{\delta\downarrow
  0}\sigma\{(z,z+\delta]\}
\\
&
=0.
\end{align*}
 Consequently,  
$$
\lim_{\delta\downarrow
  0}J(z,\delta)=0.
$$
It follows that $u$ has the right $F$-derivative given by
  \begin{equation}
 \label{e1}
\hskip-.5cm\frac{d^+u}{dF}(z)=w_\alpha^{-1}\left(\frac{d\varphi_\alpha}{dF}(z)
\int_{(l,z]}\psi_\alpha d\sigma_u+
\frac{d\psi_\alpha}{dF}(z)
\int_{(z,r)}\varphi_\alpha d\sigma_u\right).
\end{equation}
Analogous calculations yield for the left $F$-derivative
 \begin{equation}
 \label{e2}
\hskip-.5cm
\frac{d^-u}{dF}(z)=w_\alpha^{-1}\left(\frac{d\varphi_\alpha}{dF}(z)
\int_{(l,z)}\psi_\alpha d\sigma_u+
\frac{d\psi_\alpha}{dF}(z)
\int_{[z,r)}\varphi_\alpha d\sigma_u\right).
\end{equation}
Hence, we have 
 \begin{align}
\label{jump}
\nonumber
&\frac{d^-u}{dF}(z)-\frac{d^+u}{dF}(z)
\\
&\hskip2cm 
=
w_\alpha^{-1}\left(\frac{d\psi_\alpha}{dF}(z)
\varphi_\alpha(z)
-
\frac{d\varphi_\alpha}{dF}(z)
\psi_\alpha(z)\right) \sigma_u(\{z\})
\\
\nonumber
&\hskip2cm 
\geq 0
\end{align}
and this completes the proof. 
\end{proof}

\begin{remark} Choosing $F$ equal to the scale function yields (3.7)
  Corollary in  \cite{salminen85}. In fact, the idea of the proof is
  the same as in \cite{salminen85}. However, therein the proof is based
  explicitly on the Martin representation. Notice that taking $F=S$ in
 (\ref{jump}) yields 
\begin{equation}
\label{measu1}
 u^-(z)-u^+(z)= \sigma_u(\{z\}),
\end{equation}
which differs from the formula in the proof of (3.7)
  Corollary in  \cite{salminen85} due to the different normalizations of the
  representing measure in the Riesz and Martin representations. 
Notice also that taking $F(x)=x$ 
  gives, of course, a condition for the differentiablity in the usual sense.  
\end{remark}

We study next the differentiablity of 0-excessive functions. Hence, it
is assumed that $X$ is transient and, moreover, that the killing
measure  is identically zero. Then $\lim_{t\to\zeta}X_t= l\ {\rm or}\ r$ with
probability 1. As stated in Remark 1 after Theorem \ref{RDEC} the Riesz representation holds also for the 0-excessive
functions. In fact, in \cite{blumenthalgetoor68} only the case with $\alpha=0$ is
discussed in detail. The differentiablity of 0-excessive functions  can
be analyzed similarly as was done in Theorem \ref{exc-thm} for
$\alpha$-excessive functions. Therefore, we formulate the result as a corollary. For simplicity, it is assumed that the
boundary condition at a regular boundary point is killing. Then the Green function can be written as (see
\cite[p. 130]{itomckean74}, and \citep[p. 20]{borodinsalminen02})
\begin{align}
\nonumber
G_0(x,y)&=\int_0^\infty p(t;x,y)\,dt\\
\label{green0}
&=\begin{cases}
{\displaystyle\lim_{a\downarrow l,\, b\uparrow
  r}\frac{(S(x)-S(a)(S(b)-S(y))}{S(b)-S(a)},} \quad\text{if}\quad  x\leq y, \\
{\displaystyle\lim_{a\downarrow l,\,  b\uparrow
  r}\frac{(S(y)-S(a)(S(b)-S(x))}{S(b)-S(a)},} \quad\text{if}\quad  x\geq y.
\end{cases}
\end{align}

\begin{corollary}
\label{0exc}
Let $X$ be a transient diffusion as introduced above and $u$ a
0-excessive function of $X$. Assume that the scale function $S$ of $X$
is differentiable at a given point $z.$ Then $u$ has the left and the
right derivative at $z$ and it holds
\begin{equation}\label{diff-0exc}
\frac{d^- u}{dx}(z)-\frac{d^+u}{dx}(z)\geq 0.
\end{equation}
Moreover, $u$ is differentiable at $z$  if and only if $\sigma_u(\{z\})=0$.
\end{corollary}
\begin{proof}
Consider formula  (\ref{riesz2}) in case $\alpha=0$ and
$-\infty<S(l)<S(r)<+\infty. $ Since $l$ and $r$ are
assumed to be killing boundaries we have 
$$
\varphi_0(x)= S(r)-S(x)\quad {\rm and}\quad  \psi_0(x)=S(x)-S(l).
$$
Consequently, we may assume, without loss of generality, that in the
representation of $u$ in  (\ref{riesz2}) $\hat h_0\equiv 0.$ Hence, 
\begin{align}{\label{riesz0}}
&\hskip-1cm 
u(z)=\frac{S(r)-S(z)}{S(r)-S(l)}
\int_{(l,z]}(S(y)-S(l))\sigma_u(dy)
\nonumber
\\
&\hskip 3cm +\frac{S(z)-S(l)}{S(r)-S(l)}
\int_{(z,r)}(S(r)-S(y))\sigma_u(dy).
\end{align}
The  cases  $S(l)= -\infty,\, S(r)<\infty$ and  $S(l)> -\infty,\, S(r)=\infty$
can be handled similarly; we leave the details to the
reader. Formula (\ref{riesz0}) corresponds (\ref{riesz3}) and the
proof can be continued similarly as was done after
(\ref{riesz3}) but taking $F(x)=x.$  
\end{proof}

The basic assumption in Theorem \ref{exc-thm} is that  $\psi_\alpha$ 
and $\varphi_\alpha$ are $F$ -differentiable. In case $F=S$ this
assumption typically fails at points  $z\in I$ which are atoms of the
speed measure, i.e., $m(\{z\})>0$  (so called sticky points), and at
such points, see \cite[p. 129]{itomckean74}, \cite[p. 308]{R-Y}
(notice that there is a misprint in the formula in the middle of page
309; the term on the right hand side should be without factor 2) and
\cite[p.18]{borodinsalminen02},   
  \begin{equation}
 \label{eq11}
\varphi_{\alpha}^{+}(z)-\varphi_{\alpha}^{-}(z)=m(\{z\}){\cal G}\varphi_{\alpha}(z)= m(\{z\})\,\alpha\varphi_{\alpha}(z),
\end{equation}
and
  \begin{equation}
 \label{eq21}
\psi_{\alpha}^{+}(z)-\psi_{\alpha}^{-}(z)=m(\{z\}){\cal G}\psi_{\alpha}(z)=m(\{z\})\,\alpha\psi_{\alpha}(z).
\end{equation} 
Next theorem extends formula (\ref{measu1})  for diffusions having sticky points. 

\begin{thm}
\label{mass}
Let $u$ be an $\alpha$-excessive function of the diffusion $X.$ Then it holds
\begin{equation}\label{diff-mass}
u^-(z)-u^+(z)=\sigma_u(\{z\})- m(\{z\})\,\alpha\, u(z),
\end{equation}
where $u^+\, (u^-)$ denotes the right (left) derivative with respect to the scale function.
\end{thm}
\begin{proof}
Due to (\ref{eq11}) and (\ref{eq21}) we may assume without loss of generality that $u$ has the Riesz representation 
\begin{equation}\label{r1}
u(z)=\int_{(l,
  r)}G_{\alpha}(z,y)\sigma_u(dy).
\end{equation}
By similar calculations as in the proof of Theorem \ref{exc-thm} taking therein  $F\equiv S$  we obtain (cf. (\ref{e1}) and (\ref{e2}))
  \begin{equation}
 \label{e11}
u^+(z)=w_\alpha^{-1}\left(\varphi^+_\alpha(z)
\int_{(l,z]}\psi_\alpha d\sigma_u+
\psi^+_\alpha(z)
\int_{(z,r)}\varphi_\alpha d\sigma_u\right)
\end{equation}
and
 \begin{equation}
 \label{e21}
u^-(z)=w_\alpha^{-1}\left(\varphi^-_\alpha(z)
\int_{(l,z)}\psi_\alpha d\sigma_u+
\psi^-_\alpha(z)
\int_{[z,r)}\varphi_\alpha d\sigma_u\right).
\end{equation}
Subtracting (\ref{e11}) from(\ref{e21})  yields
\begin{align}\label{difference}
\nonumber 
\hskip-.5cm
u^-(z)- u^+(z)
=
&\,\omega_{\alpha}^{-1}\Big(\varphi_{\alpha}^{-}(z)-\varphi_{\alpha}^{+}(z)\Big)\int_{(l,z)}\psi_{\alpha} d\sigma_u\\
&\quad +\omega_{\alpha}^{-1}\Big(\psi_{\alpha}^{-}(z)-\psi_{\alpha}^{+}(z)\Big)\int_{(z,r)}\varphi_{\alpha} d\sigma_u\\
&
\nonumber
\qquad +\omega_{\alpha}^{-1}\psi_{\alpha}^{-}(z)\varphi_{\alpha}(z)\sigma_u(\{z\})-\omega_{\alpha}^{-1}\varphi_{\alpha}^{+}(z)\psi_{\alpha}(z)\sigma_u(\{z\}).
\end{align}
 Using (\ref{eq11}),  (\ref{eq21}) and noticing that 
\begin{align*}
\hskip-1cm\psi_{\alpha}^{-}(z)\varphi_{\alpha}(z)-\varphi_{\alpha}^{+}(z)
\psi_{\alpha}(z)&=\psi_{\alpha}^{-}(z)\varphi_{\alpha}(z)-\varphi_{\alpha}^{-}(z)
\psi_{\alpha}(z)
\\
&\hskip1cm
+\psi_{\alpha}(z)\Big(\varphi_{\alpha}^{-}(z)-\varphi_{\alpha}^{+}(z)\Big)
\\
&
=\omega_{\alpha}+\psi_{\alpha}(z)\Big(\varphi_{\alpha}^{-}(z)-\varphi_{\alpha}^{+}(z)\Big).
\end{align*}
Identity \eqref{difference} can be written as follows
\begin{align*}
u^-(z)-u^+(z)&=-\alpha\, m(\{z\})\Big(\omega_{\alpha}^{-1}\varphi_{\alpha}(z)\int_{(l,z)}\psi_{\alpha} d\sigma_u\\
&\hspace{1.5cm} +\omega_{\alpha}^{-1}\psi_{\alpha}(z)\int_{(z,r)}\varphi_{\alpha} d\sigma_u\\
&\hspace{1.5cm} +\omega_{\alpha}^{-1}\varphi_{\alpha}(z)\psi_{\alpha}(z)\sigma_u(\{z\})\Big)+\omega_{\alpha}^{-1}\omega_{\alpha}\sigma_u(\{z\})
\\
&=-\alpha\, m(\{z\})\int_{(l,r)}G_\alpha(z,y)\sigma_u(dy)+\sigma_u(\{z\})
\\
&=-\alpha\, m(\{z\})\,u(z)+\sigma_u(\{z\})
\end{align*}
by (\ref{r1}), and the proof is complete.
\end{proof}

\section{Application to optimal stopping}

\subsection{Smooth fit}

Probably the most used method to solve optimal stopping problems (with infinite
horizon) for one-dimensional diffusions is based on the principle of smooth fit. This principle says
that the value function $V$ as defined in (\ref{osp}) meets the reward
function $g$ smoothly at the boundary points of the stopping region
$\Gamma:=\{x\,:\, V(x)=g(x)\},$ i.e., $v'(x)=g'(x)$ at the boundary points in
case $g'$ exists. The idea of the method is to guess the form of
$\Gamma$ and to find its boundary points using 
 the continuity and the differentiablity of
the proposed value function. After this, a
verification theorem (see, e.g., \O ksendal \cite[p. 215 Theorem 10.4.1]{Oks}) is needed to show that the proposed value is
indeed the right one. 

In  \cite{salminen85} a criterion for the validity of the
smooth fit (with respect to the scale function) is
derived. This criterion can extracted from Theorem \ref{smoothfit1}
below by choosing $F$ equal to $S,$ the scale function. A condition for
the smooth fit with respect to ``usual'' differentiation is obtained by
taking $F$ to be the identity mapping. The criterion holds also for $\alpha=0$ (transient case with general
killing measure). We formulate the result for a left boundary point of
$\Gamma;$ obviously there is a similar result for a right boundary point.
\begin{thm}
\label{smoothfit1}
Let $z$ be a left boundary point of $\Gamma,$ i.e., 
$[z,z+\varepsilon_1)\subset \Gamma$ and $(z-\varepsilon_2,z)\subset
  \Gamma^{\,c}$for some positive $\varepsilon_1$ and  $\varepsilon_2.$
Let $F$ be a continuous and increasing function and assume that the
reward function 
$g$ and the functions $\varphi_\alpha$
and $\psi_\alpha,\ \alpha\geq 0,$ are  $F$-differentiable at $z.$ Then the value
function $V$ in (\ref{osp})  is $F$-differentiable at $z$ and the smooth fit with respect
to $F$   holds:
\begin{equation}
\label{sfeq}
\frac{d^+V}{dF}(z)= \frac{d^-V}{dF}(z)=\frac{dg}{dF}(z).
\end{equation}
\end{thm} 
\begin{proof}
Since $V>g$ on $\Gamma^c$ and $V=g$ on $\Gamma$ 
we have 
\begin{align*}
\frac{d^+V}{dF}(z)&=\lim_{\delta\to 0+}
\frac{V(z+\delta)-V(z)}{F(z+\delta)-F(z)}
=\lim_{\delta\to 0+}
\frac{g(z+\delta)-g(z)}{F(z+\delta)-F(z)}
\\
&=\frac{d^+g}{dF}(z)
=\frac{dg}{dF}(z)
\end{align*}
and
\begin{align*}
\frac{d^-V}{dF}(z)&=\lim_{\delta\to 0+}
\frac{V(z-\delta)-V(z)}{F(z-\delta)-F(z)}
=\lim_{\delta\to 0+}
\frac{V(z)-V(z-\delta)}{F(z)-F(z-\delta)}
\\
&\leq
\lim_{\delta\to 0+}
\frac{g(z)-g(z-\delta)}{F(z)-F(z-\delta)}
\\
&=
\frac{d^-g}{dF}(z)
=\frac{dg}{dF}(z).
\end{align*}
Consequently,
$$
\frac{d^-V}{dF}(z)\leq \frac{d^+V}{dF}(z), 
$$
and, hence, (\ref{diff-exc}) yields
$$
\frac{d^-V}{dF}(z)= \frac{d^-V}{dF}(z),
$$
proving the claim.
\end{proof}

Specializing to transient diffusions without killing inside the
state space and applying Corollary \ref{0exc} yields the following result which is the contents of 
Theorem 2.3 in \cite{Pes}, see also \cite{peskirshiryaev06} section 9.1.

\begin{corollary}
\label{smoothfit2}
Let $X$ be a transient diffusion as introduced 
in Corollary \ref{0exc}. Let $z$ be a point such that  $g(z)=V(z).$ 
If the reward
function $g$ and the scale function $S$ are differentiable at $z$ then the smooth fit holds at $z$:
\begin{equation}
\label{sfeq2}
\frac{d^+V}{dx}(z)= \frac{d^-V}{dx}(z)=\frac{dg}{dx}(z).
\end{equation}
\end{corollary}

\subsection{Example: Sticky Brownian motion}

In this section we study an optimal stopping problem when the underlying
process is a sticky Brownian motion with drift  $\mu\leq 0$. We let
$X=(X_t)_{t\geq 0}$ denote this process and, by definition, we take it
to be sticky at 0. The speed measure and the scale function of $X$ are
given for $\mu<0$ by
$$
m(dx)=2\e^{2\mu x}dx+2c{\varepsilon}_{\{0\}}(dx),\quad {\rm and}\quad
S(x)=\displaystyle{\frac{1}{2\mu}(1-e^{-2\mu x})},
$$
respectively, where ${\varepsilon}_{\{0\}}$ denotes the Dirac measure at
0 and the stickyness parameter $c$ is positive. In case, $\mu=0$ put
$S(x)=x,$ i.e., $X$ is in natural scale. The infinitesimal operator
associated with $X$ (see Ito and Mckean \cite[p. 111-112]{itomckean74}) is given for $x\not= 0$ by
\[
{\cal G}=\frac{1}{2}\frac{d^2}{dx^2}+\mu\frac{d}{dx}
\]
and defined by continuity at 0, that is, 
${\cal G}f(0)={\cal G}f(0+)={\cal G}f(0-).$ 
The domain is taken to be 
$$
{\cal D}:=\{f\,:\, f\in C^2_b(\R), {\cal G}f\in C_b(\R), f^{+}(0)-f^{-}(0)
=2c\,{\cal G}f(0)\}.
$$ 
Notice that in our case $S'(0)=1,$ and, hence, for
instance, $ f^+(0)=f'(0+).$  

To find the fundamental solutions $\psi_\alpha$ and $\varphi_\alpha$
associated with $X$ recall that the unique positive (up to a multiplicative
constants) increasing and decreasing solutions the ODE
$$
\frac{1}{2}u''(x)+\mu u'(x)=\alpha u(x)
$$
are given by
$$
\psi^o_\alpha(x)=\e^{(\theta-\mu)x}\quad {\rm
  and}\quad 
\varphi^o_\alpha(x)=\e^{-(\theta+\mu)x}.
$$
respectively, where $\theta:= \sqrt{2\alpha+\mu^2}.$ Consequently, we
should find constants $A, B, C,$ and $D$ such that
\begin{equation*}
{\psi}_\alpha(x):=\left\{
\begin{array}{lr}
\psi^o_\alpha(x),\quad& x\leq 0,\\
&\\
A\psi^o_\alpha(x)+B\varphi^o_\alpha(x),\quad &x\geq 0,
\end{array}\right. 
\end{equation*}
and
\begin{equation*}
{\varphi}_\alpha(x):=\left\{
\begin{array}{lr}
C\varphi^o_\alpha(x)+D\varphi^o_\alpha(x),\quad &x\leq 0,\\
&\\
\varphi^o_\alpha(x),\quad & x\geq 0,
\end{array}\right. 
\end{equation*}
are continuous (at 0) and, moreover, satisfy the condition (cf. (\ref{eq11}) and (\ref{eq21}))
\begin{equation}\label{funda-equa}
u'(0+)-u'(0-)= 2c\alpha u(0).
\end{equation}
Straightforward calculations show that 
\begin{equation}
\label{psipsi}
{\psi}_\alpha(x)=\left\{
\begin{array}{lr}
{\e}^{(\theta-\mu)x},& x\leq 0,\\
(1+\gamma)\e^{(\theta-\mu)x}-\gamma \e^{-(\theta+\mu)x},& x\geq 0,
\end{array}\right.
\end{equation}
and 
\begin{equation}
\label{phiphi}
{\varphi}_\alpha(x)=\left\{
\begin{array}{lr}
(1+\gamma)\e^{-(\theta+\mu)x}-
\gamma\e^{(\theta-\mu)x},&x\leq 0,\\
\e^{-(\theta+\mu)x},& x\geq 0,
\end{array}\right.
\end{equation}
where $\gamma:=c\alpha/\theta.$ We remark that these
expressions coincide in case $\mu=0$ with the formulas in  
\cite[p. 123]{borodinsalminen02}. 

We study now the OSP as given in (\ref{osp}) with $g(x)=(1+x)^+$
\begin{equation}\label{osp_sticky}
V(x):=\sup_{\tau\in {\cal
    M}}\E_x\left(e^{-\alpha\tau}(X_\tau+1)^+\right)=
\E_x\left(e^{-\alpha\tau^*}(X_{\tau^*}+1)^+\right),
\end{equation}
where $X$ is the sticky Brownian motion introduced above.

\medskip
\begin{proposition}
\label{BM_SBM}
 In case $\alpha=0$ the problem is equivalent with the
corresponding problem for ordinary Brownian motion with drift. The
smooth fit holds and the optimal stopping time is $\tau^*:=\inf\{t\,:\, X_t\geq (1-2|\mu|)/2|\mu|\}.$
\end{proposition}

\begin{proof} 
The value function of the problem is the smallest 0-excessive majorant
of the reward function. Recall that the Green function in case
$\alpha=0$ and there is no killing inside $I$ is determined 
by linear combinations of the scale function (see (\ref{green0}) for a
example). Since  the scale functions of $X$ and
the ordinary Brownian motion with drift are equal it follows from the
Martin representation that the
classes of $0$-excessive functions for these processes are
identical. Consequently, in the considered OSPs the value functions and the optimal stopping times are
equal. The solution of the latter problem 
was found already by Taylor
\cite{taylor68}, see also, e.g., \cite{mucci78}, \cite[p. 124-5]{shiryayev78}, and \cite{salminen85}, 
and, herefrom, it is clearly seen that the smooth fit holds.
\end{proof}
\begin{remark}
Another
explanation/proof of Proposition \ref{BM_SBM} is that making a state sticky in BM does not change the
probabilities of hitting points. Since in case $\alpha=0$ it does
not "cost to wait" the problems with or without the sticky point have the same solutions.
\end{remark}

We specialize now to case $\mu=0.$ It is proved in
\cite{croccemordecki12} for $c=1$ that the smooth fit does not hold when the 
discounting parameter $\alpha$ is in the interval
$[\alpha_1,\alpha_2),$ where $\alpha_1=(\sqrt{1+4c}-1)^2/8c$ and
  $\alpha_2=1/2.$ We wish to study this phenomenon via the
  representing measure of the value function. Consider 
 the following functions defined for $x\not= -1,\, 0$
\begin{align*}
\hskip-.5cm
&s(x):={\varphi}_\alpha(x)g'(x)-{\varphi}'_\alpha(x)g(x)
\\
&=\left\{
\begin{array}{lr}
0,& x< -1,\\
{\e}^{-\sqrt{2\alpha}\,x}\big((1+x)\sqrt{2\alpha}+1\big)&
\\
\hskip0.5cm +c\sqrt{2\alpha}
\big((1+x)\,\sqrt{2\alpha}\,\ch({\sqrt{2\alpha}\,x})+\sh({\sqrt{2\alpha}\,x})\big)
,& -1<x<0,
\\
{\e}^{-\sqrt{2\alpha}\,x}\big((1+x)\sqrt{2\alpha}+1\big),& 0<x,\\
\end{array}\right.
\end{align*}
and 
\begin{align*}
\hskip-.5cm
&t(x):=g(x){\psi}'_\alpha(x)-g'(x){\psi}_\alpha(x)
\\
&=\left\{
\begin{array}{lr}
0,& x< -1,\\
{\e}^{\,\sqrt{2\alpha}\,x}\big((1+x)\sqrt{2\alpha}-1\big),& -1<x<0,\\
{\e}^{\,\sqrt{2\alpha}\,x}\big((1+x)\sqrt{2\alpha}-1\big)&
\\
\hskip0.5cm +c\sqrt{2\alpha}
\big((1+x)\,\sqrt{2\alpha}\,\ch({\sqrt{2\alpha}\,x})-\sh({\sqrt{2\alpha}\,x})\big)
,& 0<x.
\end{array}\right.
\end{align*}
Notice that these functions are multiples of expressions in (\ref{10}) and (\ref{11}), of the Martin representing measure  if on the RHS we use $g$ instead of $u.$  It is straightforward to check the following properties of $s$ and
$t$:
\begin{description}
\item{(s1)}\quad $x\mapsto s(x)$ is decreasing for $x>-1,$
\item{(s2)}\quad $\lim_{x\to +\infty}s(x)=0,$
\item{(s3)}\quad   $\lim_{x\uparrow
  0}s(x)=\sqrt{2\alpha}+1 +2\alpha c,$\quad $\lim_{x\downarrow 0}s(x)=\sqrt{2\alpha}+1.$
\item{(t1)}\quad $x\mapsto t(x)$ is increasing for $x>-1,$
\item{(t2)}\quad  $\lim_{x\downarrow -1}t(x)=-\e^{-\sqrt{2\alpha}}<0,$\quad
  $\lim_{x\to +\infty}t(x)=+\infty,$   
\item{(t3)}\quad $\lim_{x\uparrow
  0}t(x)=\sqrt{2\alpha}-1,$\quad $\lim_{x\downarrow 0}t(x)=\sqrt{2\alpha}-1 +2\alpha c.$
\end{description}
 
\noindent Let $x^*$ denote the unique solution (if it exists) of the equation
$t(x)=0$ for $x>-1, x\not=0$; in case there is no solution we put $x^*=0.$   
Let $x_o>\max\{0,x^*\}$ and define
$$
\nu^o_g((x,+\infty]):=\frac{\psi_\alpha(x_o)}{ w_\alpha\,
 g(x_o)}\,s(x),\qquad x\geq x_o,
$$
and
\begin{equation}
\label{zero}
\nu^o_g([-\infty,x)):=
\left\{
\begin{array}{lr}
0,& x\leq x^*,\\
\displaystyle{\frac{\varphi_\alpha(x_o)}{ w_\alpha\,
 g(x_o)}\,t(x)},& x^*< x\leq x_o.
\end{array}\right.
\end{equation}
where $w_\alpha=2\sqrt{2\alpha}+2\alpha c.$ From the properties of $s$ and $t$ it is seen that these definitions
induce a Borel measure on $\R.$ Using the definition of the Wronskian
$w_\alpha$ we obtain
$$
\nu^o_g([-\infty,x_o))+\nu^o_g((x_o,+\infty])=1.
$$
Therefore, setting $\nu^o_g(\{x_o\})=0$ makes $\nu^o_g$  a 
probability measure. Notice also that
$$
\nu^o_g(\{-\infty\})=\lim_{x\to -\infty}\nu^o_g([-\infty,x))=0.
$$
and 
$$
\nu^o_g(\{+\infty\})=\lim_{x\to +\infty}\nu^o_g((x,+\infty])=0.
$$
The probability measure   $\nu^o_g$
 yields via the representation 
formula \eqref{martin-repre}  the $\alpha$-excessive function
\begin{align}
\label{Vo}
V_{o}(x):=
\left\{
\begin{array}{lr}
\displaystyle{\frac{g({x^*})}{g(x_o)\psi_\alpha({x^*})}}\,\psi_\alpha(x),& 
x\leq {x^*}.
\\
&
\\
\displaystyle{\frac{1}{g(x_o)}}\,g(x),& x\geq {x^*}.
\end{array}\right.
\end{align}
In this context  we call  $\nu^o_g$ the Martin representing
measure of $V_o.$ Clearly, the function
$x\mapsto V^*(x):=g(x_o)V_o(x)$ does not depend on $x_o.$
We conclude with the following  
\medskip
\begin{proposition}
\label{majo}
 The function $V^*$ is the value function of OSP \eqref{osp_sticky}, i.e., $V^*$  is the smallest 
$\alpha$-excessive majorant of $g.$ The optimal stopping time is
 $\tau^*=\inf\{t\,:\, X_t\geq x^*\}.$ In particular, for  $\alpha\in
 [\alpha_1,\alpha_2]$ with $\alpha_1=(\sqrt{1+4c}-1)^2/8c$ and
  $\alpha_2=1/2$ it holds that $x^*=0,$ 
\begin{equation}
\label{jumpjump}
\frac{d^-V^*}{dx}(0)-\frac{d^+V^*}{dx}(0)=\sqrt{2\alpha}-1\leq 0,
\end{equation}
and the Riesz representing measure has an atom at 0:
\begin{equation}
\label{Ratom}
\sigma_{V^\star}(\{0\})=\sqrt{2\alpha}-1 + 2\alpha c.
\end{equation}
In case, $\alpha=1/2$  the smooth fit holds and $\sigma_{V^\star}(\{0\})>0.$ If $\alpha=\alpha_1$ the
smooth fit fails and $\sigma_{V^\star}(\{0\})=0.$
\end{proposition}
\begin{proof}
By the construction, the function $V^*$ is $\alpha$-excessive. To prove
that $V^*$ is a majorant of $g$ is straightforward and elementary from the explicit
expressions. For a more sophisticated proof, notice that on $(-1,x^*)$ the function
$q(x):=g(x)/\psi_\alpha(x)$ is increasing since
$q'(x)=-t(x)/\psi^2_\alpha(x)>0$ on $(-1,x^*)$ (and for $x\not=0$ if
$0<x^*$). Consequently, for $x<x^*$
$$
\frac{g(x)}{\psi_\alpha(x)}< \frac{g(x^*)}{\psi_\alpha(x^*)}
\Leftrightarrow g(x)<V^*(x).
$$

Assume next that there exists an $\alpha$-excessive majorant $\hat V$
smaller than $V^*$. Consider first the case where the equation  $t(x)=0$ has a unique
root on $(-1,+\infty)\setminus \{ 0\} .$ We let, as above, $x^*$ denote this root. Since  $\hat V$ is assumed to be an
$\alpha$-excessive majorant of $g$ smaller than $V^*$ it holds that  $\hat V(x)=V^*(x)=g(x)$ for $x\geq x^*.$
Consequently, the Martin representing measures of  $\hat V$ and $V^*$ 
are equal on $[x^*,+\infty)$ and given by  (\ref{10}) and
  (\ref{11}). However, because $t(x^*)=0$ the representing measure
of $\hat V$   does not put mass on $[-\infty,x^*].$ Hence, the
representing measures of $\hat V$ and $V^*$ are equal and so, by the
uniqueness of the Martin representation, $\hat V=V^*.$ In case $t$ does not
have a zero on  $(-1,+\infty)\setminus \{ 0\}$  the
Martin representing measure of $V^*$ has an atom at $\{ 0\}$ given by
\begin{equation}
\label{atzero}
\nu_{V^*}(\{0\})= c^*\,t(0+),
\end{equation}
where $c^*$ is a non-negative constant given explicitly in
(\ref{constant}). Since the representing measures of $V^*$ and $\hat
V$ are equal on $(0,+\infty)$ and it is assumed that $\hat V\not\equiv
V^*$ we must have 
$$
 \nu_{V^*}(\{0\})>\nu_{\hat V}(\{0\})\geq 0
$$
and
$$
 \nu_{V^*}(\{0\})=\nu_{\hat V}([-\infty, 0])>0. 
$$
Consider now the Martin representations of $V^*$ and $\hat V$ for
$x<0:$
\begin{equation*}
V^*(x)=\int_{(-\infty,+\infty)}\frac{G_{\alpha}(x,y)}{G_{\alpha}(x_0,y)}\,\nu_{V^*}(dy)
\end{equation*}
and
\begin{equation*}
\hat V(x)=\int_{(-\infty,+\infty)}\frac{G_{\alpha}(x,y)}{G_{\alpha}(x_0,y)}\,\nu_{\hat V}(dy)
+\frac{\varphi_\alpha(x)}{\varphi_\alpha(x_o)}\,\nu_{\hat V}(\{-\infty\}),
\end{equation*}
respectively. We show that, in fact, $\hat V(x)>V^*(x)$ for all $x<0$
contradicting the assumption that $\hat V$ is smaller than $V^*.$
Indeed, for $x<0$
\begin{align*}
&\hskip-4cm\hat V(x)-V^*(x)
\\
&\hskip-3cm 
= \int_{(-\infty,x]}\frac{\psi_{\alpha}(y)\varphi_\alpha(x)}{\psi_{\alpha}(y)\varphi_\alpha(x_o)}\,\nu_{\hat V}(dy)
 +\int_{(x,0]}\frac{\psi_{\alpha}(x)\varphi_\alpha(y)}{\psi_{\alpha}(y)\varphi_\alpha(x_o)}\,\nu_{\hat
  V}(dy)
\\
\hskip 4cm 
&
+\frac{\varphi_\alpha(x)}{\varphi_\alpha(x_o)}\,\nu_{\hat
  V}(\{-\infty\}) -  \frac{\psi_{\alpha}(x)\varphi_\alpha(0)}{\psi_{\alpha}(0)\varphi_\alpha(x_o)}\,\nu_{ V^*}(\{0\}).
\end{align*}
Using that $\varphi_\alpha(0)=\psi_{\alpha}(0)=1$ and
$\varphi_\alpha(x_o)>0$ it is seen that for
$x<0$
\begin{equation}
\label{hatstar}
 \hat V(x)-V^*(x)>0
\end{equation}
is equivalent with
$$
 \varphi_\alpha(x)\,\nu_{\hat
  V}([-\infty,x)) + 
 \psi_{\alpha}(x)
 \, \int_{(x,0]} \frac{\varphi_\alpha(y)}{\psi_{\alpha}(y)}\,\nu_{ \hat V}(dy) 
-  \psi_{\alpha}(x)\,\nu_{ V^*}(\{0\})>0.
$$
Since $y\mapsto \varphi_\alpha(y)/\psi_{\alpha}(y)$ is decreasing 
(\ref{hatstar}) holds if
\begin{equation}
\label{hatstar1}
 \varphi_\alpha(x)\,\nu_{\hat
  V}([-\infty,x)) + 
 \psi_{\alpha}(x)\,\nu_{\hat
  V}((x,0])-  \psi_{\alpha}(x)\,\nu_{ V^*}(\{0\})>0.
\end{equation}
Observing that 
$$
\nu_{ V^*}(\{0\})=\nu_{\hat
  V}([-\infty,x)) + \nu_{\hat
  V}((x,0])
$$
it is seen that (\ref{hatstar1}) is true if for all $x<0$
\begin{equation*}
 \varphi_\alpha(x) - \psi_{\alpha}(x)>0,
\end{equation*}
and to check this is elementary from  (\ref{psipsi}) and  (\ref{phiphi}) or follows directly from the monotonicity. This completes the proof that $V^*$ is the smallest $\alpha$-excessive majorant of $g$.

It remains to prove (\ref{jumpjump}) and (\ref{Ratom}). Letting  $x\to
0+$ in (\ref{zero}) it is seen (cf. (\ref{atzero})) that  
\begin{equation}
\label{constant}
\nu_{V^*}(\{0\})= c^*\,t(0+)= g(x_o)\, \nu^o_g(\{0\})=\frac{\varphi_{\alpha}(x_o)}{\omega_{\alpha}}(\sqrt{2\alpha}-1 + 2\alpha c).
\end{equation}
Using \eqref{martin-riesz} we find the atom of the Riesz representing measure of $V^*$
\begin{equation*}
\sigma_{V^*}(\{0\})=\frac{1}{G_\alpha(x_o,0)}\, \nu_{V^*}(\{0\})=\sqrt{2\alpha}-1 + 2\alpha c.
\end{equation*}
It follows from \eqref{diff-mass} (and can also be checked directly
from (\ref{Vo})) that
\begin{align*}
\frac{d^-V^*}{dx}(0)-\frac{d^+V^*}{dx}(0)&=\sigma_{V^*}(\{0\})-2\alpha m(\{0\}) V^*(0)\\
&=\sqrt{2\alpha}-1 + 2\alpha c-2\alpha c\\
&=\sqrt{2\alpha}-1\leq 0,
\end{align*}
as claimed.
\end{proof}
\begin{figure}[!ht]
\begin{center}
\includegraphics[totalheight=5cm]{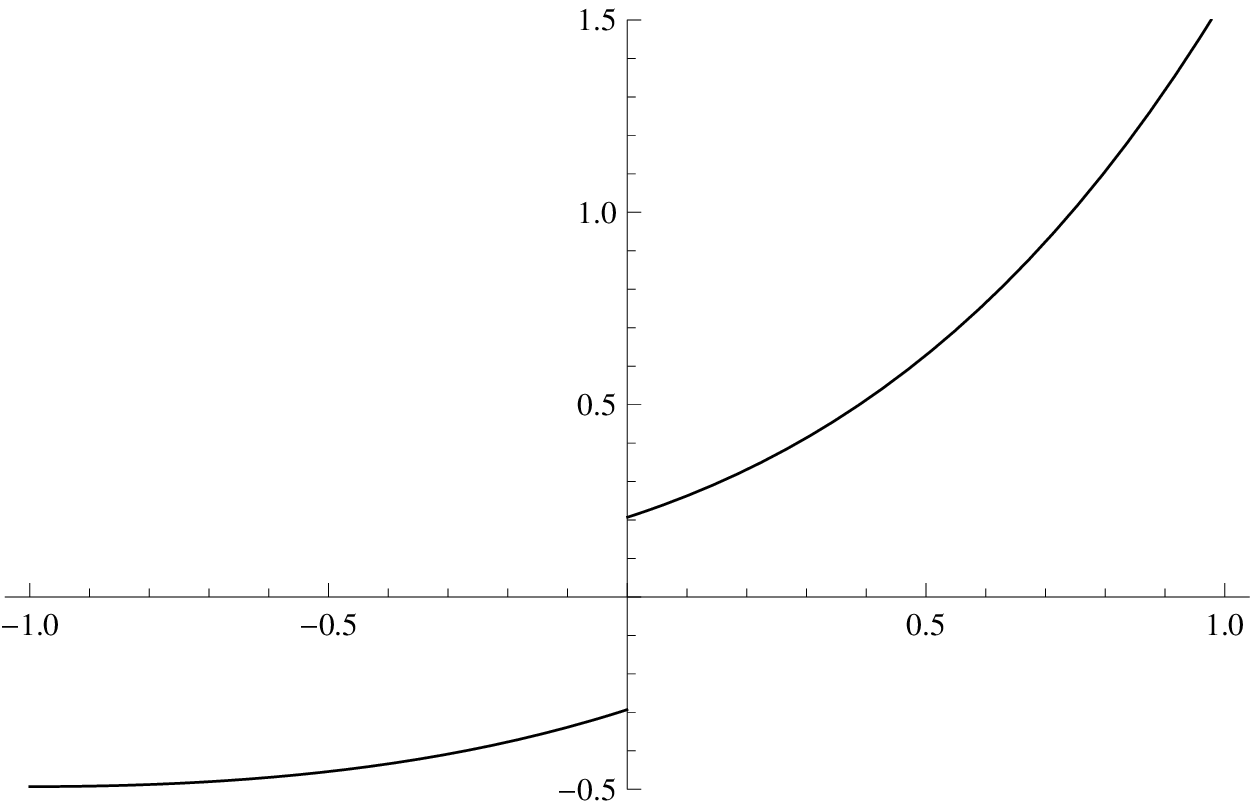}
            \caption{{\small Function $x\mapsto t(x)$, $\alpha=0.25$, $c_1=1$,  $x^*=0$.}}
\end{center}
             
\end{figure}
 
\begin{figure}[!ht]
\begin{center}
  \includegraphics[totalheight=5cm]{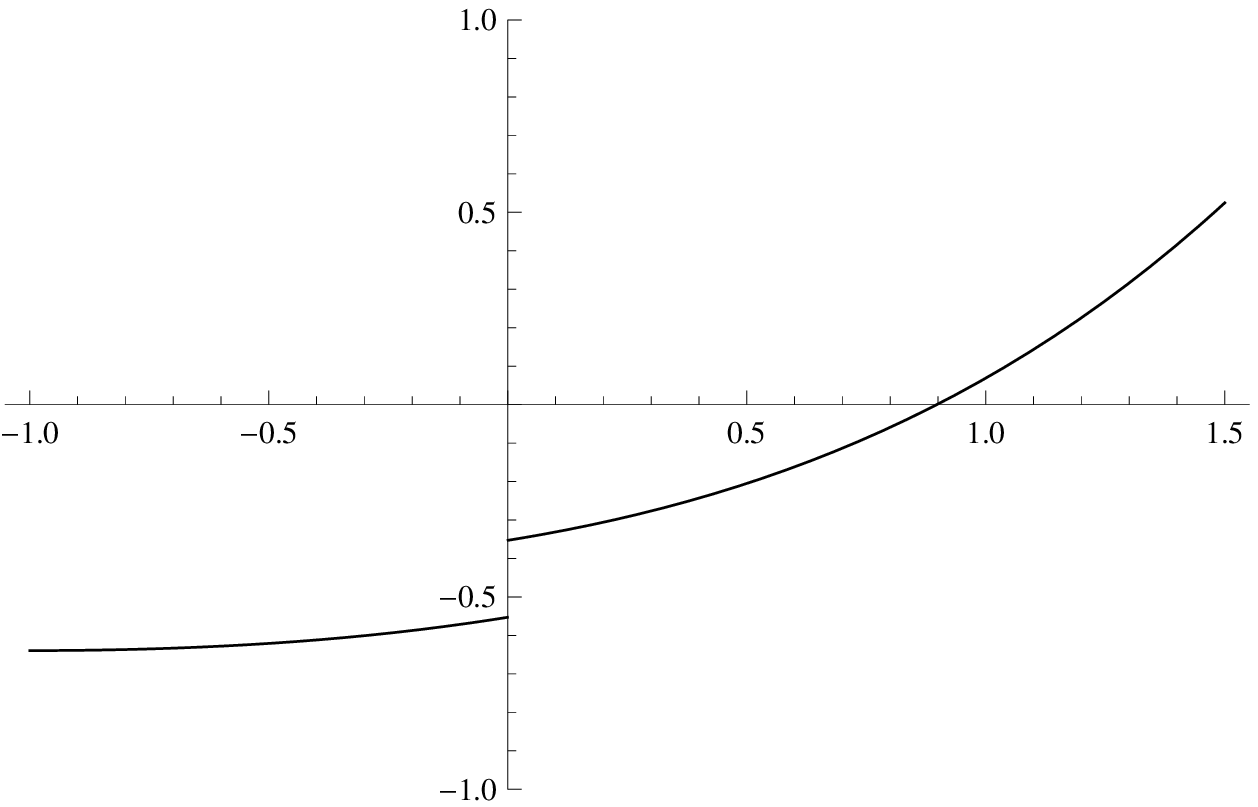}
            \caption{{\small Function $x\mapsto t(x)$, $\alpha=0.1$, $c_1=1$, $x^*>0.$}}
\end{center}
  \end{figure}
\begin{figure}[!ht]
\begin{center}
\includegraphics[totalheight=5cm]{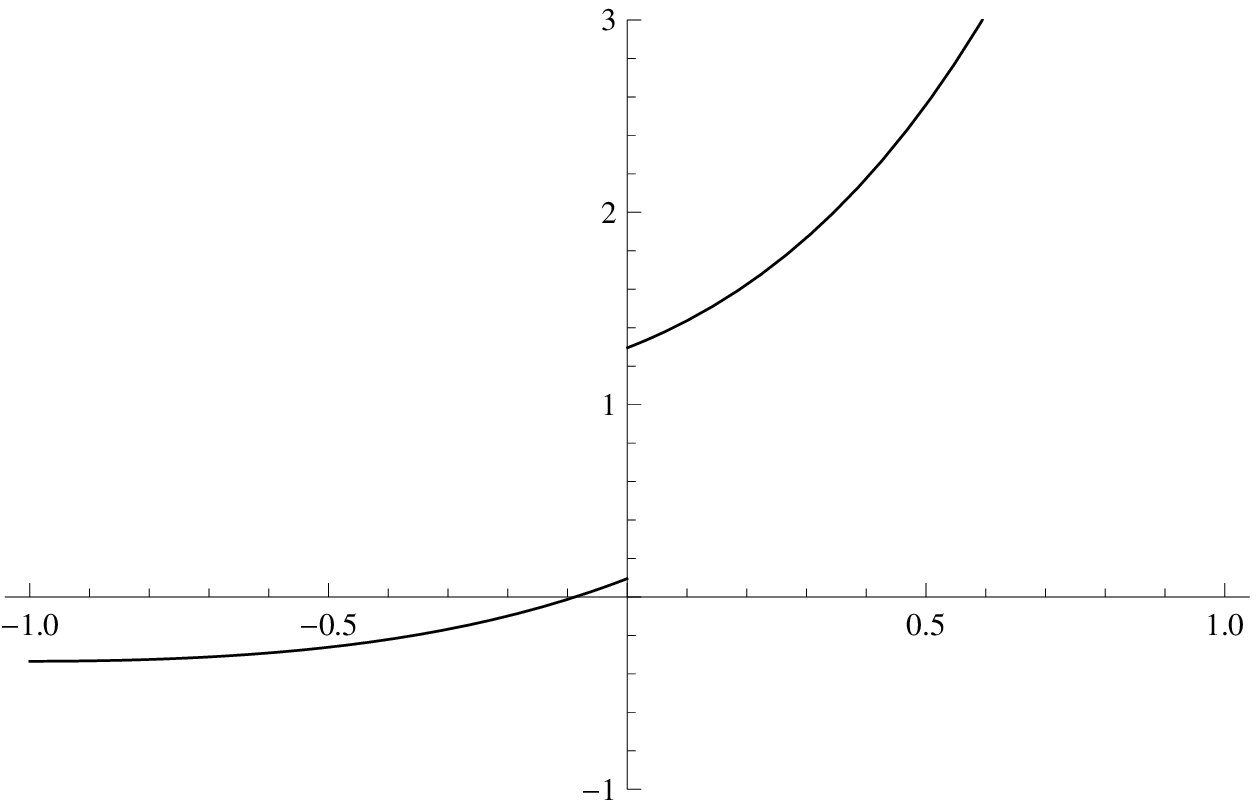}
            \caption{{\small Function $x\mapsto t(x)$, $\alpha=0.6$, $c_1=1$, $x^*<0$.}}
\end{center}
\end{figure}
\newpage

\end{document}